\documentclass[12pt,reqno,a4paper]{amsart}
\usepackage{amsfonts}
\usepackage{amsmath}

\newtheorem{theorem}{Theorem}[section]
\newtheorem{corollary}[theorem]{Corollary}
\newtheorem{lemma}[theorem]{Lemma}
\theoremstyle{definition}
\newtheorem{definition}[theorem]{Definition}
\theoremstyle{remark}

\input{tcilatex}

\begin{document}
\title[Uniformly Convex spiral function properties for Pascal distribution
series]{uniformly convex spiral functions and uniformly spirallike function
associated with Pascal distribution series}
\author{G. Murugusundaramoorthy }
\address{School of Advanced Sciences, Vellore Institute of Technology,
deemed to be university Vellore - 632014, Tamilnadu, India}
\email{gmsmoorthy@yahoo.com}
\author{B.A. Frasin}
\address{Faculty of Science, Department of Mathematics, Al al-Bayt
University, Mafraq, Jordan}
\email{bafrasin@yahoo.com}
\author{Tariq Al-Hawary}
\address{Department of Applied Science, Ajloun College, Al-Balqa Applied
University, Ajloun 26816. Jordan. }
\email{tariq\_amh@bau.edu.jo}

\begin{abstract}
The aim of this paper is to find the necessary and sufficient conditions and
inclusion relations for Pascal distribution series to be in the classes $%
\mathcal{SP}_{p}(\alpha ,\beta )$ and $\mathcal{UCV}_{p}(\alpha ,\beta )$ of
uniformly spirallike functions. Further, we consider  properties of a
special function related to Pascal distribution series. Several corollaries
and consequences of the main results are also considered.


\end{abstract}

\maketitle

\section{\protect\bigskip Introduction and definitions}

Let $\mathcal{A}$ denote the class of functions of the form%
\begin{equation}
f(z)=z+\sum_{n=2}^{\infty }a_{n}z^{n},  \label{tt}
\end{equation}%
which are analytic in the open unit disk $\mathbb{U}=\{z\in \mathbb{C}%
:\left\vert z\right\vert <1\}$ and normalized by the conditions $%
f(0)=0=f^{\prime }(0)-1.$Further, let $\mathcal{T}$ \ be a subclass of $%
\mathcal{A}$ consisting of functions of the form, 
\begin{equation}
f(z)=z-\sum\limits_{n=2}^{\infty }\left\vert a_{n}\right\vert z^{n},\qquad
z\in \mathbb{U}\text{.}  \label{m1}
\end{equation}%
A function $f\in \mathcal{A}$ is spirallike if 
\begin{equation*}
\mathfrak{R}\left( e^{-i\alpha }\frac{zf^{\prime }(z)}{f(z)}\right) >0,\ 
\end{equation*}%
for some $\alpha $ with $\left\vert \alpha \right\vert <\pi /2\ $and for all 
$z\in \mathbb{U}$ . Also $f(z)$ is convex spirallike if $zf^{\prime }(z)$ is
spirallike.

In \cite{sel}, Selvaraj and Geetha introduced the following subclasses of
unifromly spirallike and convex spirallike functions.

\begin{definition}
A function $f$ of the form (\ref{tt}) is said to be in the class $\mathcal{SP%
}_{p}(\alpha ,\beta )$ if it satisfies the following condition:%
\begin{equation*}
\mathfrak{R}\left\{ e^{-i\alpha }\left( \frac{zf^{\prime }(z)}{f(z)}\right)
\right\} >\left\vert \frac{zf^{\prime }(z)}{f^{\prime }(z)}-1\right\vert
+\beta ~\ \ \ \ \ (\left\vert \alpha \right\vert <\pi /2~;0\leq \beta <1)
\end{equation*}
\end{definition}

and $f\in $\ $\mathcal{UCV}_{p}(\alpha ,\beta )$\ if and only if $zf^{\prime
}(z)\in \mathcal{SP}_{p}(\alpha ,\beta ).$

We write%
\begin{equation*}
\mathcal{TSP}_{p}(\alpha ,\beta )=\mathcal{SP}_{p}(\alpha ,\beta )\cap 
\mathcal{T}
\end{equation*}%
and 
\begin{equation*}
\mathcal{UCT}_{p}(\alpha ,\beta )=\mathcal{UCV}_{p}(\alpha ,\beta )\cap 
\mathcal{T}\text{.}
\end{equation*}%
\bigskip

In particular, we note that $\mathcal{SP}_{p}(\alpha ,0)=\mathcal{SP}%
_{p}(\alpha )$\ and $\mathcal{UCV}_{p}(\alpha ,0)=\mathcal{UCV}_{p}(\alpha
), $ the classes of uniformly spirallike and uniformly convex spirallike
were introduced by Ravichandran et al. \cite{rav}. \bigskip For $\alpha =0,$
the classes $\mathcal{UCV}_{p}(\alpha )$ and $\mathcal{SP}_{p}(\alpha )$
respectively, reduces to the classes $\mathcal{UCV}$ and $\mathcal{SP}$
introduced and studied by Ronning \cite{ron2}.For more interesting
developments of some related subclasses of uniformly spirallike and
uniformly convex spirallike, the readers may be referred to the works of
Frasin \cite{fr,fr3}, Goodman \cite{goo1,goo2}, Tariq Al-Hawary and Frasin 
\cite{fr2}, Kanas and Wisniowska \cite{kan1, kan2} and Ronning \cite{ron1,
ron2}.

A function $f\in \mathcal{A}$ is said to be in the class $\mathcal{R}^{\tau
}(A,B)$,$\tau \in \mathbb{C}\backslash \{0\}$, $-1\leq B<A\leq 1,$ if it
satisfies the inequality%
\begin{equation*}
\left\vert \frac{f^{\prime }(z)-1}{(A-B)\tau -B[f^{\prime }(z)-1]}%
\right\vert <1,\ \ \ \;z\in \mathbb{U}.
\end{equation*}

\bigskip This class was introduced by Dixit and Pal \cite{dix}.

A variable $x$ is said to be \emph{Pascal distribution} if it takes the
values $0,1,2,3,\dots $ with probabilities $(1-q)^{m}$, $\dfrac{qm(1-q)^{m}}{%
1!}$, $\dfrac{q^{2}m(m+1)(1-q)^{m}}{2!}$, $\dfrac{q^{3}m(m+1)(m+2)(1-q)^{m}}{%
3!}$, \dots respectively, where $q$ and $m$ are called the parameters, and
thus 
\begin{equation*}
P(x=k)=\binom{k+m-1}{m-1}\,q^{k}(1-q)^{m},k=0,1,2,3,\dots .
\end{equation*}

Very recently, El-Deeb \cite{pascal} introduced a power series whose
coefficients are probabilities of Pascal distribution 
\begin{equation}  \label{PHI}
\Psi_q^m(z)=z+\sum\limits_{n=2}^{\infty} \binom{n+m-2}{m-1}\cdot
q^{n-1}(1-q)^mz^n,\qquad z\in \mathbb{U}
\end{equation}
where $m \geq 1;0\leq q\leq 1$ and we note that, by ratio test the radius of
convergence of above series is infinity. We also define the series 
\begin{equation}  \label{PSI}
\Phi_q^m(z) =2z-\Psi_q^m(z)=z-\sum\limits_{n=2}^{\infty} \binom{n+m-2}{m-1}%
\cdot q^{n-1}(1-q)^mz^n, \qquad z\in \mathbb{U}.
\end{equation}
Now, we considered the linear operator 
\begin{equation*}
\mathcal{I}_q^m(z):\mathcal{A}\rightarrow\mathcal{A}
\end{equation*}
defined by the convolution or Hadamard product 
\begin{equation}  \label{I}
\mathcal{\ I}_q^mf(z)=\Psi_q^m(z)\ast f(z) =z+\sum\limits_{n=2}^{\infty} 
\binom{n+m-2}{m-1}\cdot q^{n-1}(1-q)^m a_nz^n,\qquad z\in \mathbb{U}
\end{equation}

Motivated by several earlier results on connections between various
subclasses of analytic and univalent functions by using hypergeometric
functions (see for example,\cite{1,fr4,2,8,9}) and by the recent
investigations (see for example, \cite{ou, fra, por1, por3, por2, mur1,mur2}%
), in the present paper we determine the necessary and sufficient conditions
for $\Phi _{q}^{m}(z)$ to be in our classes $\mathcal{TSP}_{p}(\alpha ,\beta
)$ and $\mathcal{UCT}_{p}(\alpha ,\beta )$ and connections of these
subclasses with $\mathcal{R}^{\tau }(A,B)$. Finally, we give conditions for
the function $\mathcal{G}_{q}^{m}f(z)=\int_{0}^{z}\frac{\Phi _{q}^{m}(t)}{t}%
dt$ \ belonging to the above classes.

To establish our main results, we need the following Lemmas.

\begin{lemma}
\label{lem1}\cite{sel} A function\ $f$\ of the form (\ref{m1}) is in $%
\mathcal{TSP}_{p}(\alpha ,\beta )$ if and only if it satisfies%
\begin{equation}
\sum\limits_{n=2}^{\infty }(2n-\cos \alpha -\beta )\left\vert
a_{n}\right\vert \leq \cos \alpha -\beta ~~\ \ \ \ \ (\left\vert \alpha
\right\vert <\pi /2~;0\leq \beta <1).  \label{t1}
\end{equation}
\end{lemma}

\textit{In particular, when }$\beta =0,$ \textit{we obtain a necessary and
sufficient condition for a function }$f$\textit{\ of the form (\ref{m1}) to
be in the class }$\mathcal{TSP}_{p}(\alpha )$\textit{\ is that }%
\begin{equation}
\sum\limits_{n=2}^{\infty }(2n-\cos \alpha )\left\vert a_{n}\right\vert \leq
\cos \alpha ~~\ \ \ \ \ (\left\vert \alpha \right\vert <\pi /2).  \label{t2}
\end{equation}

\begin{lemma}
\label{lem2}\cite{sel} A function\ $f$\ of the form (\ref{m1}) is in $%
\mathcal{UCT}_{p}(\alpha ,\beta )$ if and only if it satisfies%
\begin{equation}
\sum\limits_{n=2}^{\infty }n(2n-\cos \alpha -\beta )\left\vert
a_{n}\right\vert \leq \cos \alpha -\beta ~~\ \ \ \ \ (\left\vert \alpha
\right\vert <\pi /2~;0\leq \beta <1).  \label{b1}
\end{equation}
\end{lemma}

\textit{In particular, when }$\beta =0,$ \textit{we obtain a necessary and
sufficient condition for a function }$f$\textit{\ of the form (\ref{m1}) to
be in the class }$\mathcal{UCT}_{p}(\alpha )$\ is that%
\begin{equation}
\sum\limits_{n=2}^{\infty }n(2n-\cos \alpha )\left\vert a_{n}\right\vert
\leq \cos \alpha ~~\ \ \ \ \ (\left\vert \alpha \right\vert <\pi /2).
\label{b3}
\end{equation}

\begin{lemma}
\label{lem3}\cite{dix} If $f$\textit{\ }$\in $\textit{\ }$\mathcal{R}^{\tau
}(A,B)$ is of the form \textit{(\ref{tt})} , then%
\begin{equation*}
\left\vert a_{n}\right\vert \leq (A-B)\frac{\left\vert \tau \right\vert }{n}%
,\ \ \ \ \ \ n\in \mathbb{N}\backslash \{1\}\text{.}
\end{equation*}
\end{lemma}

\textit{The result is sharp.}

\section{The necessary and sufficient conditions}

For convenience throughout in the sequel, we use the following identities
that hold at least for $m\geq 2$ and $0\leq q<1$: 
\begin{eqnarray}
\sum\limits_{n=0}^{\infty }\binom{n+m-1}{m-1}\cdot q^{n} &=&\frac{1}{%
(1-q)^{m}} \\
\sum\limits_{n=0}^{\infty }\binom{n+m-2}{m-2}\cdot q^{n} &=&\frac{1}{%
(1-q)^{m-1}} \\
\sum\limits_{n=0}^{\infty }\binom{n+m}{m}\cdot q^{n} &=&\frac{1}{(1-q)^{m+1}}
\\
\sum\limits_{n=0}^{\infty }\binom{n+m+1}{m+1}\cdot q^{n} &=&\frac{1}{%
(1-q)^{m+2}}
\end{eqnarray}%
By simple calculation we get the following: 
\begin{equation*}
\sum\limits_{n=2}^{\infty }\binom{n+m-2}{m-1}\,q^{n-1}=\sum\limits_{n=0}^{%
\infty }\binom{n+m-1}{m-1}\,q^{n}-1
\end{equation*}%
\begin{equation*}
\sum\limits_{n=2}^{\infty }(n-1)\binom{n+m-2}{m-1}\,q^{n-1}=q~m\sum%
\limits_{n=0}^{\infty }\binom{n+m}{m}\,q^{n}.
\end{equation*}%
and 
\begin{equation*}
\sum\limits_{n=2}^{\infty }(n-1)(n-2)\binom{n+m-2}{m-1}%
\,q^{n-1}=q^{2}~m(m+1)\sum\limits_{n=0}^{\infty }\binom{n+m+1}{m+1}\,q^{n}.
\end{equation*}%
%
%
%
%
Unless otherwise mentioned, we shall assume in this paper that $\left\vert
\alpha \right\vert <\pi /2,$ $0\leq \beta <1$ while $m\geq 1$ and $0\leq q<1$%
.

First we obtain the necessary and sufficient conditions for $\Phi _{q}^{m}$
to be in the class $\mathcal{TSP}_{p}(\alpha ,\beta ).$

\begin{theorem}
\label{thmd2} \label{th1}We have $\Phi _{q}^{m}\in \mathcal{TSP}_{p}(\alpha
,\beta )$ if and only if 
\begin{equation}
\frac{2q~m}{1-q}+(2-\cos \alpha -\beta )\left[ 1-(1-q)^{m}\right] \leq \cos
\alpha -\beta .  \label{d2}
\end{equation}
\end{theorem}

\begin{proof}
Since%
\begin{equation}
\Phi_q^m(z)=z-\sum\limits_{n=2}^{\infty }\binom{n+m-2}{m-1}\cdot
q^{n-1}(1-q)^mz^{n}  \label{n5}
\end{equation}%
in view of Lemma \ref{lem1}, it suffices to show that%
\begin{equation}
\sum\limits_{n=2}^{\infty }(2n-\cos \alpha -\beta )\binom{n+m-2}{m-1}\cdot
q^{n-1}(1-q)^m\leq \cos \alpha -\beta .  \label{n1}
\end{equation}%
Writing 
\begin{equation*}
n=(n-1)+1
\end{equation*}
in (\ref{n1}) we have 
\begin{eqnarray*}
&&\sum\limits_{n=2}^{\infty }(2n-\cos \alpha -\beta )\binom{n+m-2}{m-1}\cdot
q^{n-1}(1-q)^m \\
&=&\sum\limits_{n=2}^{\infty }(2(n-1)+2-\cos \alpha -\beta )\binom{n+m-2}{m-1%
}\cdot q^{n-1}(1-q)^m \\
&=&2\sum\limits_{n=2}^{\infty }(n-1)\binom{n+m-2}{m-1}\cdot q^{n-1}(1-q)^m
+(2-\cos \alpha -\beta )\sum\limits_{n=2}^{\infty }\binom{n+m-2}{m-1}\cdot
q^{n-1}(1-q)^m \\
&=&2q~m(1-q)^m \sum\limits_{n=0}^{\infty}\binom{n+m}{m}\,q^{n}+(2-\cos
\alpha -\beta )(1-q)^m \left[\sum\limits_{n=0}^{\infty}\binom{n+m-1}{m-1}%
\,q^{n}-1\right] \\
&=&\frac{2q~m}{1-q}+(2-\cos \alpha -\beta )\left[1-(1-q)^m\right].
\end{eqnarray*}

But this last expression is bounded above by $\cos \alpha -\beta ~$if and
only if~(\ref{d2}) holds.
\end{proof}

\begin{theorem}
\label{th11}We have $\Phi _{q}^{m}\in $ $\mathcal{UCT}_{p}(\alpha ,\beta )$
if and only if 
\begin{equation}
\frac{2q^{2}~m(m+1)}{(1-q)^{2}}+(6-\cos \alpha -\beta )\frac{q~m}{1-q}%
+(2-\cos \alpha -\beta )\left[ 1-(1-q)^{m}\right] \leq \cos \alpha -\beta .
\label{lp}
\end{equation}
\end{theorem}

\begin{proof}
In view of Lemma \ref{lem2}, we must show that%
\begin{equation}
\sum\limits_{n=2}^{\infty }n(2n-\cos \alpha -\beta )\binom{n+m-2}{m-1}\cdot
q^{n-1}(1-q)^m\leq \cos \alpha -\beta .  \label{we}
\end{equation}%
Writing%
\begin{equation*}
n=(n-1)+1
\end{equation*}%
and%
\begin{equation*}
n^{2}=(n-1)(n-2)+3(n-1)+1
\end{equation*}

in (\ref{we}) 
\begin{eqnarray*}
&&\sum\limits_{n=2}^{\infty }n(2n-\cos \alpha -\beta )\binom{n+m-2}{m-1}%
\cdot q^{n-1}(1-q)^me^{-m} \\
&=&\sum\limits_{n=2}^{\infty }(2(n-1)(n-2)+6(n-1))\binom{n+m-2}{m-1}\cdot
q^{n-1}(1-q)^m \\
&+&\sum\limits_{n=2}^{\infty }(n-1)(-\cos \alpha -\beta )\binom{n+m-2}{m-1}%
\cdot q^{n-1}(1-q)^m \\
&&+\sum\limits_{n=2}^{\infty }(2-\cos \alpha -\beta )\binom{n+m-2}{m-1}\cdot
q^{n-1}(1-q)^m \\
&=&2q^2~m(m+1)(1-q)^m\sum\limits_{n=0}^{\infty}\binom{n+m+1}{m+1}%
\,q^{n}+(6-\cos \alpha -\beta )q~m (1-q)^m\sum\limits_{n=0}^{\infty}\binom{%
n+m}{m}\,q^{n} \\
&+&(2-\cos \alpha -\beta )(1-q)^m\left[\sum\limits_{n=0}^{\infty}\binom{n+m-1%
}{m-1}\,q^{n}-1\right] \\
&=&2q^2~m(m+1)(1-q)^m\frac{1}{(1-q)^{m+2}}+(6-\cos \alpha -\beta )q~m (1-q)^m%
\frac{1}{(1-q)^{m+1}} \\
&+&(2-\cos \alpha -\beta )(1-q)^m\left[\frac{1}{(1-q)^{m}}-1\right] \\
&=&\frac{2q^2~m(m+1)}{(1-q)^{2}}+(6-\cos \alpha -\beta ) \frac{q~m}{1-q}%
+(2-\cos\alpha -\beta )\left[1-(1-q)^{m}\right] \\
\end{eqnarray*}

Therefore, we see that the last expression is bounded above by $\cos \alpha
-\beta ~$if (\ref{lp}) is satisfied.
\end{proof}

\section{Inclusion Properties}

Making use of Lemma \ref{lem3}, we will study the action of the Pascal
distribution series on the class $\mathcal{TSP}_{p}(\alpha ,\beta ).$

\begin{theorem}
\label{th2}Let $m>1.~$If $\ f\in \mathcal{R}^{\tau }(A,B),\ $then $\mathcal{I%
}_{q}^{m}f(z)$ $\in $ $\mathcal{TSP}_{p}(\alpha ,\beta )$ if 
\begin{equation}
(A-B)|\tau |\left[ \frac{{}}{{}}2\left[ 1-(1-q)^{m}\frac{{}}{{}}\right] -%
\frac{\cos \alpha +\beta }{q(m-1)}\left[ (1-q)-(1-q)^{m}-q(m-1)(1-q)^{m}%
\right] \right] \allowbreak \leq \cos \alpha -\beta .  \label{d3}
\end{equation}
\end{theorem}

\begin{proof}
In view of Lemma \ref{lem1}, it suffices to show that%
\begin{equation*}
\sum\limits_{n=2}^{\infty }(2n-\cos \alpha -\beta )\binom{n+m-2}{m-1}\cdot
q^{n-1}(1-q)^m\left\vert a_{n}\right\vert \leq \cos \alpha -\beta .
\end{equation*}

Since $f\in \mathcal{R}^{\tau }(A,B),$ then by Lemma \ref{lem3}, we have%
\begin{equation}
\left\vert a_{n}\right\vert \leq \frac{(A-B)\left\vert \tau \right\vert }{n}.
\label{vv}
\end{equation}

Thus, we have 
\begin{eqnarray*}
&&\sum\limits_{n=2}^{\infty }(2n-\cos \alpha -\beta )\binom{n+m-2}{m-1}\cdot
q^{n-1}(1-q)^{m}\left\vert a_{n}\right\vert \\
&\leq &(A-B)\left\vert \tau \right\vert \left[ \sum\limits_{n=2}^{\infty }%
\frac{1}{n}(2n-\cos \alpha -\beta )\binom{n+m-2}{m-1}\cdot q^{n-1}(1-q)^{m}%
\right] \\
&=&(A-B)\left\vert \tau \right\vert (1-q)^{m}\left[ 2\sum\limits_{n=2}^{%
\infty }\binom{n+m-2}{m-1}\cdot q^{n-1}-(\cos \alpha +\beta
)\sum\limits_{n=2}^{\infty }\frac{1}{n}\binom{n+m-2}{m-1}\cdot q^{n-1}\right]
\\
&=&(A-B)|\tau |(1-q)^{m}\left[ 2\left[ \sum\limits_{n=0}^{\infty }\binom{%
n+m-1}{m-1}\,q^{n}-1\right] \right. \\
&-&\left. \frac{\cos \alpha +\beta }{q(m-1)}\left[ \sum_{n=0}^{\infty }%
\binom{n+m-2}{m-2}\,q^{n}-1-(m-1)q\right] \right] \\
&=&(A-B)|\tau |\left[ \frac{{}}{{}}2\left[ 1-(1-q)^{m}\frac{{}}{{}}\right] -%
\frac{\cos \alpha +\beta }{q(m-1)}\left[ (1-q)-(1-q)^{m}-q(m-1)(1-q)^{m}%
\right] \right]
\end{eqnarray*}%
But this last expression is bounded by $\cos \alpha -\beta $, if (\ref{d3})
holds. This completes the proof of Theorem \ref{th2}.
\end{proof}

Applying Lemma \ref{lem2} and using the same technique as in the proof of
Theorem \ref{th2}, we have the following result.

\begin{theorem}
\label{th4}If $\ f\in \mathcal{R}^{\tau }(A,B),\ $then $\mathcal{I}_{q}^{m}f$
is in $\mathcal{UCT}_{p}(\alpha ,\beta )$ if 
\begin{equation}
(A-B)|\tau |\left( \frac{2q~m}{1-q}+(2-\cos \alpha -\beta )\left[ 1-(1-q)^{m}%
\right] \right) \leq \cos \alpha -\beta .
\end{equation}
\end{theorem}

\section{Properties of a special function}

\begin{theorem}
\label{th3} If the function $\mathcal{G}_{q}^{m}$ is given by 
\begin{equation}
\mathcal{G}_{q}^{m}(z):=\int_{0}^{z}\frac{\Phi _{q}^{m}(t)}{t}dt,\;z\in 
\mathbb{U},  \label{pl}
\end{equation}%
then $\mathcal{G}_{q}^{m}\in \mathcal{UCT}_{p}(\alpha ,\beta )$ if and only
if the inequality (\ref{d2}) holds.
\end{theorem}

\begin{proof}
Since%
\begin{equation*}
\mathcal{G}_{q}^{m}(z)=z-\sum_{n=2}^{\infty }\binom{n+m-2}{m-1}%
\,q^{n-1}(1-q)^{m}\frac{z^{n}}{n},\;z\in \mathbb{U},
\end{equation*}%
then by Lemma \ref{lem2}, we need only to show that%
\begin{equation*}
\sum\limits_{n=2}^{\infty }n(2n-\cos \alpha -\beta )\times \frac{1}{n}\binom{%
n+m-2}{m-1}\,q^{n-1}(1-q)^{m}\leq \cos \alpha -\beta ,
\end{equation*}

or, equivalently%
\begin{equation}
\sum\limits_{n=2}^{\infty }(2n-\cos \alpha -\beta )\binom{n+m-2}{m-1}\cdot
q^{n-1}(1-q)^m\leq \cos \alpha -\beta .  \label{gg}
\end{equation}
The remaining part of the proof of Theorem \ref{th3} is similar to that of
Theorem \ref{th1}, and so we omit the details.
\end{proof}

\bigskip

\begin{theorem}
\label{th6}If $m>1,$ then the function $\mathcal{G}_{q}^{m}\mathcal{\ }\in 
\mathcal{TSP}_{p}(\alpha ,\beta )$ if and only if ~%
\begin{equation*}
\frac{{}}{{}}2\left[ 1-(1-q)^{m}\frac{{}}{{}}\right] -\frac{\cos \alpha
+\beta }{q(m-1)}\left[ (1-q)-(1-q)^{m}-q(m-1)(1-q)^{m}\right] \leq \cos
\alpha -\beta .
\end{equation*}
\end{theorem}

The proof of Theorem \ref{th6} is lines similar to the proof of Theorem \ref%
{th3}, so we omitted the proof of Theorem \ref{th6}.

\section{Corollaries and consequences}

By specializing the parameter $\beta =0$ in Theorems \ref{th1}-\ref{th6}, we
obtain the following corollaries.

\begin{corollary}
We have $\Phi _{q}^{m}\in $ $\mathcal{TSP}_{p}(\alpha )$ if and only if 
\begin{equation}
\frac{2q~m}{1-q}+(2-\cos \alpha )\left[ 1-(1-q)^{m}\right] \leq \cos \alpha .
\label{hh}
\end{equation}
\end{corollary}

\begin{corollary}
We have $\Phi _{q}^{m}\in $ $\mathcal{UCT}_{p}(\alpha )$ if and only if 
\begin{equation}
\frac{2q^{2}~m(m+1)}{(1-q)^{2}}+(6-\cos \alpha )\frac{q~m}{1-q}+(2-\cos
\alpha )\left[ 1-(1-q)^{m}\right] \leq \cos \alpha .
\end{equation}
\end{corollary}

\begin{corollary}
Let $m>1\mathit{\ }.$ If $\ f\in \mathcal{R}^{\tau }(A,B),\ $then $\mathcal{I%
}_{q}^{m}f$ $\in \mathcal{TSP}_{p}(\alpha )$ if 
\begin{equation}
(A-B)|\tau |\left[ \frac{{}}{{}}2\left[ 1-(1-q)^{m}\frac{{}}{{}}\right] -%
\frac{\cos \alpha }{q(m-1)}\left[ (1-q)-(1-q)^{m}-q(m-1)(1-q)^{m}\right] %
\right] \allowbreak \leq \cos \alpha .
\end{equation}
\end{corollary}

\begin{corollary}
If $\ f\in \mathcal{R}^{\tau }(A,B),\ $then $\mathcal{I}_{q}^{m}f$ $\in $ $%
\mathcal{UCT}_{p}(\alpha )$ if 
\begin{equation}
(A-B)|\tau |\left( \frac{2q~m}{1-q}+(2-\cos \alpha )\left[ 1-(1-q)^{m}\right]
\right) \leq \cos \alpha .
\end{equation}
\end{corollary}

\begin{corollary}
\bigskip The function $\mathcal{G}_{q}^{m}\mathcal{\ }\in $ $\mathcal{UCT}%
_{p}(\alpha )$ if and only if ~the inequality (\ref{hh}) holds.
\end{corollary}

\begin{corollary}
If $m>1,$ then the function $\mathcal{G}_{q}^{m}\mathcal{\ }\in $ $\mathcal{%
TSP}_{p}(\alpha )$ if and only if ~%
\begin{equation*}
\frac{{}}{{}}2\left[ 1-(1-q)^{m}\frac{{}}{{}}\right] -\frac{\cos \alpha }{%
q(m-1)}\left[ (1-q)-(1-q)^{m}-q(m-1)(1-q)^{m}\right] \leq \cos \alpha .
\end{equation*}
\end{corollary}

\end{document}